\documentclass[12pt,twoside]{article}
\usepackage{amssymb,amsmath,amsthm,latexsym}
\usepackage{amsfonts}
\usepackage{enumerate}
\usepackage{graphicx}

\newtheorem{thm}{Theorem}[section]

\newtheorem{cor}[thm] {Corollary}
\newtheorem{lem} [thm]{Lemma}

\theoremstyle{definition} % Definitions, examples, remarks, algorithms should be in roman type, not italic.

\newtheorem{rmk}[thm] {Remark}
\newtheorem{defn}[thm]{Definition}
\numberwithin{equation}{section}

\begin{document}
	\label{'ubf'}
	\setcounter{page}{1} %Put here the starting page number
	
	\markboth {\hspace*{-9mm} \centerline{\footnotesize \sc
			% Put here the left page top label
			On the Powers of Signed Graphs}
	}
	{ \centerline {\footnotesize \sc
			%put here the author's name
			Shijin T V, Germina K A, Shahul Hameed K} \hspace*{-9mm}
	}
	\begin{center}
		{
			{\huge \textbf{On the Powers of Signed Graphs
					% Put the title of the paper here
				}
			}
			
			\bigskip
			Shijin T V \footnote{\small Department of Mathematics, Central University of Kerala, Kasaragod - 671316,\ Kerala,\ India.\ Email: shijintv11@gmail.com}
			Germina K A \footnote{\small  Department of Mathematics, Central University of Kerala, Kasaragod - 671316,\ Kerala,\ India.\ Email: srgerminaka@gmail.com}
			Shahul Hameed K \footnote{\small  Department of
				Mathematics, K M M Government\ Women's\ College, Kannur - 670004,\ Kerala,  \ India.  E-mail: shabrennen@gmail.com}
			
		}
\end{center}
\newcommand\spec{\operatorname{Spec}}
\thispagestyle{empty}
\begin{abstract}
A signed graph is an ordered pair $\Sigma=(G,\sigma),$ where $G=(V,E)$ is the underlying graph of $\Sigma$ with a signature function $\sigma:E\rightarrow \{1,-1\}$. 
In this article, we define $n^{th}$ power of a signed graph and discuss some properties of these powers of signed graphs. As we can define two types of signed graphs as the power of a signed graph, necessary and sufficient conditions are given for an $n^{th}$ power of a signed graph to be unique. Also, we characterize balanced power signed graphs.
\end{abstract} 

\textbf{Key Words:} Signed graph, Signed distance, Distance compatibility, Power of signed graphs.

\textbf{Mathematics Subject Classification (2010):}   05C12, 05C22, 05C50.
%\newpage
%\tableofcontents
%\setcounter{tocdepth}{3}
\section{Introduction}
In this paper, we will treat only simple, finite and connected signed graphs. A signature on a graph $G=(V,E)$ is a function $\sigma:E\rightarrow \{1,-1\}$. A signed graph is a graph $G=(V,E)$ with a signature $\sigma$ and is denoted as $\Sigma=(G,\sigma)$, where $G$ is called the underlying graph of $\Sigma$. The sign of a cycle in a signed graph is the product of the signs of its edges. A cycle in a signed graph is said to be balanced if it has positive sign. A signed graph $\Sigma$ is said to be balanced if all of its cycles are positive and $\Sigma$ is unbalanced, otherwise \cite{hrry}.

To begin with, we recall  and adopt some definitions and notations from ~\cite{sdist}, in which the  concept of signed distance in signed graphs and distance compatible signed graphs are introduced.
 
Let $u$ and $v$ be any two vertices in a connected graph $G$. As usual, $d(u,v)$ denotes the distance (the lenght of the shortest path) between $u$ and $v$. 
Let $\mathcal{P}_{(u,v)}$ denote the collection of all shortest paths $P_{(u,v)}$ between them. Then, the distance between $u$ and $v$ in a signed graph is defined as:

$d_{\max}(u,v) = \sigma_{\max}(u,v) d(u,v)=\max\{\sigma(P_{(u,v)}): P_{(u,v)} \in \mathcal{P}_{(u,v)} \}d(u,v)$ and 

$d_{\min}(u,v) = \sigma_{\min}(u,v) d(u,v)=\min\{\sigma(P_{(u,v)}): P_{(u,v)} \in \mathcal{P}_{(u,v)} \}d(u,v),$ where the sign of a path $P$ in $\Sigma$ is defined as $\sigma(P)=\prod_{e\in E(P)} \sigma(e).$\\
Two vertices $u$ and $v$ in  $\Sigma$ are said to be \emph{distance-compatible} (briefly, \emph{compatible}) if $d_{\min}{(u,v)}=d_{\max}{(u,v)}.$ A signed graph $\Sigma$ is said to be (distance)-compatible or simply compatible, if every pair of vertices is compatible and $\Sigma$ is incompatible, otherwise.

Corresponding to the functions $d_{\max}$ and $d_{\min}$, there are two types of distance matrices in a signed graph called signed distance matrices \cite{sdist} as given below.

	\par(D1)  $D^{\max}(\Sigma)=(d_{\max}(u,v))_{n\times n}$.
	\par(D2)  $D^{\min}(\Sigma)=(d_{\min}(u,v))_{n\times n}$.
	
Also, the concept of associated signed complete graphs associated with $D^{\max}(\Sigma)$ and $D^{\min}(\Sigma)$ are introduced in \cite{sdist}, as follows. 

\begin{defn}
	The associated signed complete graph $K^{D^{\max}}(\Sigma)$ with respect to $D^{\max}(\Sigma)$ is obtained by  joining the non-adjacent vertices of $\Sigma$ with edges having signs 
	\begin{equation*}
	\sigma(uv)= \sigma_{\max}(uv)
	\end{equation*}
	
	The associated signed complete graph $K^{D^{\min}}(\Sigma)$ with respect to $D^{\min}(\Sigma)$ is obtained by joining the non-adjacent vertices of $\Sigma$ with edges having signs 
	\begin{equation*}
	\sigma(uv)= \sigma_{\min}(uv)
	\end{equation*}
\end{defn}
Whenever, $D^{\max}=D^{\min}=D^\pm$, say, the associated signed complete graph of $\Sigma$ is denoted by $K^{D^\pm}(\Sigma).$

The concept of $n^{th}$ power of graph is discussed in \cite{spow}. The $n^{th}$ power of a graph $G=(V,E)$ is denoted as $G^n$ and is defined as, graph having the same vertex set as that of $G$ and  any two vertices $u$ and $v$ are adjacent in $G^n$ if their distance $d(u,v)$ is less than or equal to $n$.

 In this paper, we define $n^{th}$ power of a signed graph by using the concept of signed distance in signed graphs and discuss some properties of $n^{th}$ power of signed graphs. Also, we characterize the balanced power signed graphs.

\section{Main Results}
Corresponding to $\sigma_{\max}$ and $\sigma_{\min}$, two types of  $n^{th}$ powers of signed graph for a given signed graph $\Sigma$, can be defined as follows.
\begin{defn}
	Let $\Sigma=(G,\sigma)$ be a signed graph.
	\par(D1)  The $n^{th}$ power signed graph $\Sigma_{\max}^{n}$ is a signed graph $\Sigma_{\max}^{n}=(G^n,\sigma'),$ where $G^n$ is the underlying graph of $\Sigma_{\max}^{n}$ and for any edge $e=uv\in G^n,\ \sigma'(uv)=\sigma_{\max}(u,v).$   
	\par(D2)  The $n^{th}$ power signed graph $\Sigma_{\min}^{n}$ is a signed graph $\Sigma_{\min}^{n}=(G^n,\sigma''),$ where $G^n$ is the underlying graph of $\Sigma_{\min}^{n}$ and for any edge $e=uv\in G^n,\ \sigma''(uv)=\sigma_{\min}(u,v).$   
	
\end{defn}

\begin{rmk}
	\rm{The $n^{th}$ power of a signed graph $\Sigma$ is said to be unique whenever,  $\Sigma_{\max}^{n}=\Sigma_{\min}^{n},$ and is denoted by $\Sigma^{n}$.}	
\end{rmk}

The following Theorem \ref{T1} gives a characterization for the uniqueness of $n^{th}$ power of a signed graph.

\begin{thm}\label{T1}
	Let $\Sigma=(G,\sigma)$ be a signed graph. Then, the $n^{th}$ power of $\Sigma$ is unique if and only if there exists no incompatible pair of vertices at a distance less than or equal to $n$.
\end{thm}
\begin{proof}
	Suppose that the $n^{th}$ power of $\Sigma$ is unique. If possible, assume on the contrary that there exists an incompatible pair $u$ and $v$ in $\Sigma$ at a distance, say, $k\leq n$. Then, there exists two $uv$ paths $P$ and $Q$ of distance $k$ in $\Sigma$ with $\sigma(P)$ is positive and $\sigma(Q)$ is negative. While considering $\Sigma^n,$ all the vertices with distance less than or equal to $n$ will have to be joined by  an edge with a unique signature. But here, the path $P$ will induce a positive edge $uv$ in $\Sigma_{\max}^{n}$ and the path $Q$ will induce a negative edge $uv$ in $\Sigma_{\min}^{n},$ implies $\Sigma_{\max}^{n}\neq \Sigma_{\min}^{n},$ a contradiction.\\
	Hence, there exists no incompatible pair of vertices in $\Sigma$ at a distance less than or equal to $n$.
	
	Conversely, suppose that there exists no incompatible pair of vertices in $\Sigma$ at a distance less than or equal to $n.$ Then, for any two vertices $u$ and $v$ with distance less than or equal to $n$ in $\Sigma,$ $\sigma_{\max}(u,v)=\sigma_{\min}(u,v)=\sigma(u,v).$ That is, $\Sigma_{\max}^{n}=\Sigma_{\min}^{n}=\Sigma^{n}.$ Hence, the $n^{th}$ power of $\Sigma$ is unique.
	\end{proof}

\begin{thm}
	Let $\Sigma=(G,\sigma)$ be a signed graph with diameter less than or equal to $n$. Then, the $n^{th}$ power signed graph $\Sigma_{\max}^{n}=(G^n,\sigma')(\text{or}\ \Sigma_{\min}^{n}=(G^n,\sigma''))$ is the associated signed complete graph $K^{D^{\max}}(\Sigma)(\text{or}\ K^{D^{\min}}(\Sigma))$. Moreover, if $\Sigma$ is compatible then, $\Sigma^n=(G^n,\sigma')$ is the associated signed complete graph $K^{D^{\pm}}(\Sigma).$  
\end{thm}
\begin{proof}
	Let $\Sigma=(G,\sigma)$ be a signed graph with diameter less than or equal to $n.$ That is, the maximum distance between any two vertices in $\Sigma$ is $n.$ Therefore, while taking $\Sigma_{\max}^{n}$ all the non-adjacent vertices $u$ and $v$ will form an edge $uv$ with sign $\sigma'(uv)=\sigma_{\max}(u,v).$ Hence, $\Sigma_{\max}^{n}$ is the associated signed complete graph $K^{D^{\max}}(\Sigma).$ Similarly, if we consider  $\Sigma_{\min}^{n}$ all the non-adjacent vertices $u$ and $v$ will form an edge $uv$ with sign $\sigma''(uv)=\sigma_{\min}(u,v).$ Hence, $\Sigma_{\min}^{n}$ is the associated signed complete graph $K^{D^{\min}}(\Sigma).$
	
	If $\Sigma$ is compatible, then $\Sigma_{\max}^{n}=\Sigma_{\min}^{n}=\Sigma^{n}.$ Since, the diameter is less than or equal to $n,$ all the non-adjacent vertices in $\Sigma$ will form an edge in $\Sigma^{n}$ and sign of these edges are $\sigma'(uv)=\sigma_{\max}(u,v)=\sigma_{\min}(u,v).$ Hence, $\Sigma^n$ is the associated signed complete graph $K^{D^{\pm}}(\Sigma).$
\end{proof}

\begin{lem}\label{L1}
	Let $\Sigma=(G,\sigma)$ be a signed graph and $u,v$ be two vertices in $\Sigma$. Then, for any $u$--$v$ path $P$ of length $k$ there exists a $u$--$v$ path $P'$ of length $\lceil\frac{k}{n}\rceil$ in $\Sigma^n=(G^n,\sigma'),$ where $\sigma(P)=\sigma'(P').$
\end{lem}
\begin{proof}
	Let $\Sigma=(G,\sigma)$ be a signed graph and let $P$ be a $u$--$v$ path of length $k$ in $\Sigma.$ In $\Sigma^n$ all the vertices with distance at most $n$ in $\Sigma$ will form an edge. By considering division algorithm on $k$ and $n,$  we get $k=nq+r,\text{where}\ 0\leq r< n.$\\
	\textbf{Case 1:} $r=0.$ \\
	Then, $k=nq.$ Hence, there will be $q$ edges between $u$ and $v$ in $\Sigma^n.$\\
	\textbf{Case 2:} $r\neq 0.$\\ 
	Then, $k=nq+r.$ Since, $r<n$ the path of length $r$ will form an edge in $\Sigma^n.$ Then, there will be $q+1$ edges between $u$ and $v$ in $\Sigma^n.$\\
	From the above cases, it can be concluded that, corresponding to the path $P$ of length $k$ in $\Sigma$ there is a path $P'$ from $u$ to $v$ in $\Sigma^n$ of length $\lceil\frac{k}{n}\rceil.$\\
	To prove $\sigma(P)=\sigma'(P').$ By the definition of $\Sigma^n$ any path $P$ of length $k\leq n$ will form an edge $e$ in $\Sigma^n$ and here $\sigma(P)=\sigma'(e).$ \\
    Let $k> n$ and $k=nq+r$ and let $\{e_1, e_2, \dots e_{nq+r}\}$ be the edge set of $P.$ Then, $\sigma(P)= \prod_{i=1}^{nq+r}\sigma(e_i).$\\
    Suppose that $r\neq 0.$ Then, $P$ can be written as the union of edge disjoint paths, $P_i'=e_{(i-1)n+1}, e_{(i-1)n+2}, \dots, e_{in},\ 1\leq i\leq q$ and $P_{q+1}'=e_{nq+1}, e_{nq+2}, \dots, e_{nq+r}$ whose length is at most $n.$  Since, every path of length at most $n$ will form an edge in $\Sigma^n,$  $\{e_1',e_2', \dots, e_{q+1}'\}$ be the edges in $\Sigma^n$ corresponding to paths $P_1', \dots, P_{q+1}'.$ Let $P'$ be the path from $u$ to $v$ with edges $e_1',e_2', \dots, e_{q+1}'$ in $\Sigma^n.$ \\
    Then,
	$\sigma'(P')=\prod_{i=1}^{q+1}\sigma'(e_i')=\prod_{i=1}^{q+1}\sigma(P_i')=\sigma(P).$ When $r=0,$ in a similar way we can see that $\sigma(P)=\sigma'(P').$ 
\end{proof}	

\begin{lem}\label{Le}
	Let $\Sigma=(G,\sigma)$ be a signed graph and $u,v$ be two vertices in $\Sigma.$ Then, for any $u$--$v$path $P$ of length $k$ in $\Sigma^n=(G^n,\sigma'),$ there exists a $u$--$v$ path $P'$ of length $k',$ where $(k-1)n+1 \leq k' \leq kn$ in $\Sigma$ such that $\sigma'(P)=\sigma(P').$
\end{lem}
\begin{proof}
	Let $\Sigma^n=(G^n,\sigma')$ be the $n^{th}$ power of $\Sigma=(G,\sigma)$ and $P$ be a $uv$ path of length $k$ in $\Sigma^n.$\\
	\textbf{Case 1:} If $k=1$.\\
	Then, $uv$ is an edge in $\Sigma^n.$ By the definition, each edge in $\Sigma^n$ corresponds to a path $P'$ of length $k' \leq n$ in $\Sigma,$ where $\sigma'(uv)=\sigma(P')$.\\
	\textbf{Case 2:} If $k>1.$\\
	Let $\{e_1, e_2, \dots, e_k\}$ be the edge set of $P$ in $\Sigma^n$, where each edge $e_i,$ $1 \leq i \leq k$ corresponds to a path $P_i$ of length $l_i\leq n$ in $\Sigma.$ Then, the concatenation $P'=\cup_i P_i$ will form a $uv$ path of length $k'=\sum_{i=1}^{k}l_i \leq kn$ in $\Sigma$ and the sign of $P'$ is given by $\sigma(P')=\prod_{i=1}^{k}\sigma(P_i)=\prod_{i=1}^{k}\sigma'(e_i)=\sigma'(P).$ \\
	Since, $P$ is of length $k$ by Lemma \ref{L1} $k'$ should be greater than $(k-1)n.$
\end{proof}
\begin{thm}
	Let $\Sigma=(G,\sigma)$ be a signed graph with diameter greater than $n$ and the $n^{th}$ power of $\Sigma$ is unique. Then, $\Sigma^n=(G^n,\sigma')$ is compatible implies $\Sigma$ is compatible.
\end{thm}
\begin{proof}
	Suppose that $\Sigma$ is incompatible. Since, the $n^{th}$ power of $\Sigma$ is unique, by Theorem \ref{T1}, there exist no incompatible pair of vertices at a distance less than or equal to $n.$ Let $u$ and $v$ be an incompatible pair of vertices at a distance $k>n.$ Then, there exist two shortest path $P$ and $Q$ from $u$ to $v$ of length $k$ with $\sigma(P)$ is positive and $\sigma(Q)$ is negative. Then, by Lemma \ref{L1}, there exists two paths $P'$ and $Q'$ from $u$ to $v$ in $\Sigma^{n}$ of length $\lceil\frac{k}{n}\rceil$ with $\sigma(P)=\sigma'(P')$ and $\sigma(Q)=\sigma'(Q').$ Thus, $u$ and $v$ will form an incompatible pair of vertices in $\Sigma^{n},$ a contradiction. Hence, the signed graph $\Sigma$ is compatible. 
\end{proof}

\begin{rmk}
	\rm{The converse of the above theorem is not generally true. For example, consider the cycle $C_7^-$ and its square signed graph given in Figure \ref{Fone}. The cycle $C_7^-$ is compatible, where its square signed graph contains incompatible vertices $u_1$ and $u_4$.}
\end{rmk}
\begin{figure}[h]
	\centering
	\includegraphics[width=12cm]{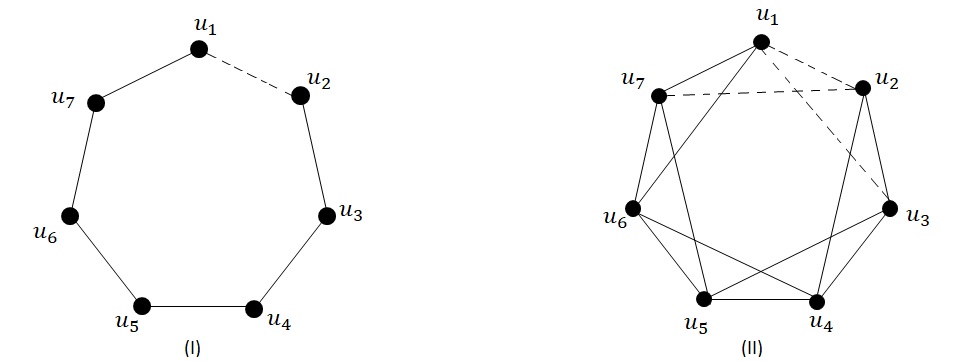}
	\caption{The signed graph $C_7^-$ and its square signed graph.}
	\label{Fone}
\end{figure}

\subsection{Balance criterion in the power of a signed graph}

In this section we give a characterization for balance in the $n^{th}$ power of a signed graph. First we recall some results from ~\cite{sdist, sdpsg}.
\begin{thm} [\cite{sdist}]\label{nbc}
	For a signed graph $\Sigma$ the following statements are equivalent:
	\begin{enumerate}[\rm{(1)}]
		\item [\rm{(1)}] $\Sigma$ is balanced
		\item [\rm{(2)}] The associated signed complete graph $K^{D^{\max}}(\Sigma)$ is balanced.
		\item [\rm{(3)}] The associated signed complete graph $K^{D^{\\min}}(\Sigma)$ is balanced.
		\item [\rm{(4)}] $D^{\max}(\Sigma)=D^{\min}(\Sigma)$ and the associated signed complete graph $K^{D^{\pm}}(\Sigma)$ is balanced.
	\end{enumerate}	
\end{thm}

\begin{thm} [\cite{sdist}] \label{sgs}
	A signed graph $\Sigma$ is balanced if and only if  the associated signed complete graph $K^{D^{\pm}}(\Sigma)$ has the spectrum $\begin{pmatrix}  n-1 & -1\\ 1 & n-1
	\end{pmatrix}$.
\end{thm}

\begin{lem} [\cite{sdpsg}] \label{L2}
		Let $u$ and $v$ be incompatible pair of vertices with least distance in a $2$-connected non-geodetic signed graph. Then, there will be two internally disjoint shortest paths from $u$ to $v$ of opposite signs.
\end{lem}

\begin{lem}\label{blcm}
	Let $\Sigma=(G,\sigma)$ be a $2$-connected signed graph. Then, $\Sigma$ is balanced implies $\Sigma^n=(G^n,\sigma')$ is compatible.
\end{lem}
\begin{proof}
	Let $\Sigma=(G,\sigma)$ be a $2$-connected balanced signed graph. If possible, let $u$ and $v$ be an incompatible pair of vertices in $\Sigma^n.$ Since, $\Sigma^n$ is $2$-connected by using Lemma \ref{L2}, we get two internally disjoint shortest $uv$ paths $P$ and $Q$ of opposite signs. Then, by Lemma \ref{Le}, we can find two $uv$ paths $P'$ and $Q'$ in $\Sigma$ with $\sigma(P')=\sigma'(P)$ and $\sigma(Q')=\sigma'(Q).$ Therefore, $P'$ and $Q'$ can not be the same. If $P'$ and $Q'$ are internally disjoint, then the concatenation $(P')\cup (Q')^{-1}$ will be a negative cycle in $\Sigma,$ a contradiction. Suppose that $P'$ and $Q'$ are not internally disjoint. Consider the cycles $C_1, C_2, \dots, C_m$ formed by the common points of $P'$ and $Q'.$ Since, $P'$ and $Q'$ are not the same, there exists at least one such cycle. Also, since $\sigma(P')\neq \sigma(Q'),$ we can find at least one cycle among $C_1, C_2, \dots, C_m$, say $C_i$ with common points $u_i$ and $v_i$ of $P'$ and $Q',$ where the sign of $u_iv_i$ path along $P'$ and along $Q'$ are distinct. Then, the cycle $C_i$ will be a negative cycle in $\Sigma,$ a contradiction. Hence, $\Sigma^n$ should be compatible.  
\end{proof}

\begin{lem}\label{L3}
Let $\Sigma=(G,\sigma)$ be a signed graph. Then, the associated signed complete graphs $K^{D^{\max}}(\Sigma)=K^{D^{\max}}(\Sigma_{\max}^n)$ and $K^{D^{\min}}(\Sigma)=K^{D^{\min}}(\Sigma_{\min}^n).$ Moreover, if $\Sigma$ is balanced, then $K^{D^{\pm}}(\Sigma)=K^{D^{\pm}}(\Sigma^n).$
\end{lem}
\begin{proof}
Let $\Sigma=(G,\sigma)$ be a signed graph and $\Sigma_{\max}^n=(G^n,\sigma')$ be the $n^{th}$ power of $\Sigma.$ Then, by Lemma \ref{L1}, corresponding to the shortest $uv$ path $P$ in $\Sigma$, there is a shortest $uv$ path $P'$ in $\Sigma_{\max}^n$ where,
$\sigma_{\max}(P_{(u,v)})=\sigma'_{\max}(P'_{(u,v)}).$ The associated signed complete graph $K^{D^{\max}}(\Sigma)$ is otained by joining all the non-adjacent vertices in $\Sigma,$ with edges having signs $\sigma(u,v)=\sigma_{\max}(u,v).$ Similarly, the associated signed complete graph $K^{D^{\max}}(\Sigma_{\max}^n)$ is otained by joining all the non-adjacent vertices in $\Sigma_{\max}^n,$ with edges having signs $\sigma'(u,v)=\sigma'_{\max}(u,v)=\sigma_{\max}(u,v).$\\
Also, if $u$ and $v$ are adjacent in $\Sigma,$ we have $\sigma'(u,v)=\sigma_{\max}(u,v).$ Hence, $K^{D^{\max}}(\Sigma)=K^{D^{\max}}(\Sigma_{\max}^n).$ Similarly, we get $K^{D^{\min}}(\Sigma)=K^{D^{\min}}(\Sigma_{\min}^n).$\\
If $\Sigma$ is balanced, then $\Sigma_{\max}^n=\Sigma_{\min}^n=\Sigma^n$ and by Lemma \ref{blcm} $\Sigma^n$ is compatible. Therefore, $K^{D^{\max}}(\Sigma_{\max}^n)=K^{D^{\min}}(\Sigma_{\min}^n)=K^{D^{\pm}}(\Sigma^n).$ Also, $\Sigma$ is balanced implies $\sigma(u,v)=\sigma_{\max}(u,v)=\sigma_{\min}(u,v)$ and $K^{D^{\max}}(\Sigma)=K^{D^{\min}}(\Sigma)=K^{D^{\pm}}(\Sigma).$ Hence, $K^{D^{\pm}}(\Sigma)=K^{D^{\pm}}(\Sigma^n).$
\end{proof}

\begin{thm} \label{cbp}
	A $2$-connected signed graph $\Sigma=(G,\sigma)$ is balanced if and only if $\Sigma^n=(G^n,\sigma')$ is balanced.
\end{thm}
\begin{proof}
	Suppose that $\Sigma$ is balanced. Then, by Theorem \ref{nbc} we get $D^{\max}(\Sigma)=D^{\min}(\Sigma)$ and the associated signed complete graph $K^{D^{\pm}}(\Sigma)$ is balanced.
	Since, $\Sigma$ is balanced by using Lemma \ref{blcm} and Lemma \ref{L3} we get $\Sigma^n$ is compatible and $K^{D^{\pm}}(\Sigma)=K^{D^{\pm}}(\Sigma^n).$ Which implies, $D^{\max}(\Sigma^n)=D^{\min}(\Sigma^n)$ and $K^{D^{\pm}}(\Sigma^n)$ is balanced. Again, by using Theorem \ref{nbc}, we get $\Sigma^n$ is balanced. 
	
	Conversely, suppose that $\Sigma^n$ is balanced. Being a subgraph of $\Sigma^n,$ $\Sigma$ should be balanced.
\end{proof}

The following Corollary is an immediate consequence of Theorem \ref{sgs} and Theorem \ref{cbp}.
\begin{cor}
	Let $\Sigma=(G,\sigma)$ be a $2$-connected compatible signed graph of order $m$. Then, the $n^{th}$ power signed graph $\Sigma^n$ is balanced if and only if the associated signed complete graph $K^{D^{\pm}}(\Sigma)$ has the spectrum $\begin{pmatrix}  m-1 & -1\\ 1 & m-1
	\end{pmatrix}$.
\end{cor}

\section*{Acknowledgement}

The first author would like to acknowledge his gratitude to the Council of Scientific and Industrial Research (CSIR), India, for the financial support under the CSIR Junior Research Fellowship scheme, vide order nos.: 09/1108(0032)/2018-EMR-I.

 % Authors extend their thanks to the reviewer(s) for the critical comments and insightful suggestions which improved the content and presentation style of the content of the paper.

\end{document}